\documentclass{article}
\pdfoutput=1
\usepackage[english]{babel}
\usepackage[percent]{overpic}
\usepackage{graphics}
\usepackage{float}
\usepackage{caption}
\usepackage{adjustbox}
\usepackage{mathtools}
\usepackage{graphicx} 
\usepackage{cite}
\usepackage{amssymb}
\usepackage{amsmath}
\usepackage{amsthm}
\usepackage{graphicx}
\usepackage{appendix}
\usepackage{cases}
\usepackage{color}
\usepackage[ruled,linesnumbered]{algorithm2e}
\usepackage{algorithmic}
\usepackage{tikz}
\usetikzlibrary{external}
\newtheorem{theorem}{Theorem}[section]

\newtheorem{problem}{Problem}[section]

\newtheorem{proposition}{Proposition}[section]
\usepackage{amsfonts}
\usepackage{color,xcolor}

\usepackage{multirow}
\usepackage{soul}
\theoremstyle{remark}
\newtheorem{remark}[theorem]{Remark}

\usepackage[colorlinks=true, allcolors=blue]{hyperref}
\usepackage{appendix}
\usepackage{booktabs}
\usepackage{multirow}
\usepackage{cases}
\usepackage{array}

\usepackage{newtxtext,newtxmath}
\usepackage{makecell}
\usepackage{indentfirst}
\hypersetup{colorlinks = true, urlcolor = blue}
\usepackage{subfigure}

\usepackage[letterpaper,top=2cm,bottom=2cm,left=3cm,right=3cm,marginparwidth=1.75cm]{geometry}

\title{A Jacobi Field Approach to Splitting Detection \\ in Schr\"{o}dinger Bridge}
\author{
 \normalsize{
Chunhai Jiao$^{1,2,}$\footnotemark[2]\,,
Jin Guo$^{6}$\footnotemark[2]\,,
Haoyan Zhang$^{1,2}$\footnotemark[2]\,,
Jinqiao Duan$^{4,5,}$\footnotemark[3]\,,
Ting Gao$^{1,2,3}$\footnotemark[1]\
}\\[10pt]
\footnotesize{$^1$ School of Mathematics and Statistics, Huazhong University of Science and Technology, Wuhan, China} \\
\footnotesize{$^2$ Center for Mathematical Science, Huazhong University of Science and Technology, Wuhan, China} \\
\footnotesize{$^3$ Steklov-Wuhan Institute for Mathematical Exploration, Huazhong University of Science and Technology, China} \\
\footnotesize{$^4$ Department of Mathematics and Department of Physics, Great Bay University, Dongguan, China} \\
\footnotesize{$^5$ Guangdong Provincial Key Laboratory of Mathematical and Neural dynamic Systems, Dongguan, China. }\\
\footnotesize{$^6$ Department of Mathematics, City University of Hong Kong, Hong Kong, China.}}

\begin{document}

\maketitle

\begin{abstract}
We study the problem of detecting the onset of path splitting in stochastic interpolation between probability distributions. This question is especially subtle when the target distribution is nonconvex or supported on disconnected components, where interpolating trajectories may separate into distinct branches. Motivated by the stochastic control and Schrödinger bridge viewpoint, we propose a Jacobi field based indicator for identifying candidate splitting times and locations. Our approach is based on the Jacobi field associated with the linearization of an induced interpolating flow. Starting from a stochastic interpolation ansatz, we construct an Eulerian velocity field by conditional averaging and derive its spatial Jacobian in terms of the local posterior geometry of the target sample cloud. This allows us to interpret the symmetric part of the Jacobian as a local strain tensor and to use its spectral structure to quantify the amplification of infinitesimal perturbations along reference trajectories. Numerical experiments on nonconvex and disconnected target distributions show that the proposed indicator consistently localizes the emergence of branching regions and captures the temporal development of splitting. These results suggest that Jacobi field analysis provides a natural mathematical framework for studying local instability and splitting phenomena in stochastic interpolation.
\end{abstract}

\section{Introduction}
Detecting trajectory splitting in stochastic interpolation~\cite{JMLR:v26:23-1605} between probability distributions is an important problem in the analysis of noisy transport~\cite{doi:10.1137/16M1061382}, especially when the target distribution is nonconvex or supported on disconnected components~\cite{lei2019mode}. In contrast to the classical rare-event setting, where one seeks a dominant transition pathway through large-deviation or path-integral principles, distribution-to-distribution transport involves an entire evolving family of random trajectories, whose branching geometry is considerably more subtle. Classical approaches based on Freidlin–Wentzell theory~\cite{kifer1988random} and path-integral methods~\cite{huang2021estimating} provide powerful descriptions of rare transitions, while the Schrödinger bridge~\cite{schrodinger1932theorie} offers a natural variational framework for stochastic interpolation by linking stochastic control~\cite{chen2021stochastic}, entropy minimization~\cite{chen2016relation}, and optimal transport~\cite{leonard2013survey}. This connection makes the analysis of trajectory geometry, including splitting phenomena, particularly relevant in contemporary machine learning applications.

Generative models \cite{goodfellow2014generative} are designed to approximate the latent distributions of real-world data and synthesize high-quality novel samples, facilitating the evolution of artificial intelligence from perception to creation. A cornerstone in this domain was the advent of Generative Adversarial Networks (GANs) \cite{goodfellow2014generative}. GANs utilize an adversarial training mechanism between a generator and a discriminator, enabling the generation of realistic samples without explicit density estimation. Despite their transformative impact, GANs are notoriously difficult to train and exhibit high sensitivity to hyperparameters. More critically, they frequently suffer from mode collapse \cite{kossale2022mode}, a phenomenon where the generator captures only a limited subset of the true distribution modes while neglecting others \cite{goodfellow2020generative}. From a theoretical point of view \cite{brenier1987polar, figalli2010regularity}, this issue is intrinsically linked to the discontinuity of optimal transport maps. 

To mitigate these limitations, Flow Matching (FM)\cite{lipman2022flow} has emerged as a compelling alternative to diffusion-based paradigms. By regressing a deterministic vector field that drives a probability density path from a prior to the data distribution, FM offers a simulation-free training objective that is both stable and efficient. The efficacy of FM in combating mode collapse is deeply rooted in Optimal Transport (OT) theory\cite{villani2008optimal}. Gu et al.~\cite{an2019ae} introduced a generative framework based on Extended Semi-Discrete Optimal Transport (AE-OT). Their theoretical analysis shows that mode collapse and sample mixing in GANs~\cite{goodfellow2014generative} and VAEs~\cite{kingma2013auto} stem from discontinuities in optimal transport maps under non-convex data distributions, while deep neural networks inherently represent continuous mappings. 

When the target distribution is nonconvex or supported on disconnected components, a stochastic interpolation that preserves the multi-modal structure of the target need not remain dynamically coherent throughout the entire evolution. This naturally motivates the question of whether, and when, the interpolating family of trajectories develops branch-like separation at intermediate times. Identifying the onset of this splitting~\cite{farina2020splitting} is a central question, since it marks the moment when the dynamics starts to encode the multi-branch geometry of the target law. Beyond its mathematical significance, such a detection mechanism may also be useful in applications where one seeks early indicators of structural transition. For instance, in data driven progression models from healthy to diseased populations, the detected splitting time may serve as a proxy for the onset of pathological divergence.

Our contributions are summarized as follows:
\begin{itemize}
    \item  We formulate splitting detection as a variational and geometric problem for stochastic interpolation between probability distributions. We introduce a particle level stochastic interpolation ansatz and derive the induced Eulerian velocity field by conditional averaging. Motivated by the classical Jacobi field in Riemannian geometry, we then establish the associated first order variational equations for nearby interpolating trajectories, thereby identifying the local velocity gradient as the fundamental object governing infinitesimal trajectory separation.
    \item Based on this variational structure, we develop a Jacobian based criterion for detecting local path splitting. In particular, we show that the symmetric part of the induced velocity gradient plays the role of a local strain tensor, and that its largest eigenvalue provides a computable measure of the strongest instantaneous stretching direction. This leads to a practical indicator for identifying candidate splitting times and locations. 
    \item We provide numerical evidence that the proposed Jacobi field criterion captures the emergence of splitting in stochastic interpolation toward nonconvex and disconnected targets. Specifically, by tracking the largest eigenvalue of the symmetric part of the induced velocity gradient, we observe that particles with the highest stretching intensity progressively concentrate near the visually emerging branching regions.
\end{itemize}

The remainder of this paper is organized as follows. Section 2 introduces the theoretical foundations underlying our approach. Section 3 formally defines the main research problem. Section 4 details the proposed methodology, including the solution framework and corresponding algorithms. Section 5 presents experimental results across diverse datasets, demonstrating the effectiveness and robustness of our method. Finally, Section 6 concludes with a summary of key findings and outlines directions for future research.

\section{Preliminaries}
Optimal Transport (OT) \cite{peyre2019computational,solomon2018optimal} provides a principled mathematical framework for comparing probability distributions. Tracing its origins to Monge’s 1781 transportation problem \cite{rachev1985monge}, the theory was later relaxed by Kantorovich to establish its modern foundation \cite{ambrosio2003optimal}. OT formulates the objective as finding a joint probability measure that minimizes transport cost subject to mass conservation constraints, yielding the Wasserstein distance as a geometric metric between distributions. A fundamental link between OT and stochastic processes is the Schrödinger Bridge (SB) problem, originally proposed by Schrödinger in the 1930s \cite{schrodinger1932theorie}. Emerging from a thought experiment in statistical mechanics regarding the likely evolution of Brownian particles between fixed initial and final states, the SB problem is mathematically equivalent to an entropy-regularized form of classical OT. Instead of solely minimizing transport cost, the SB seeks a probability law on paths that minimizes the relative entropy with respect to a reference Wiener process, while satisfying marginal constraints \cite{nutz2021introduction}.

Interest in the Schrödinger Bridge is driven by multiple theoretical and practical perspectives. Schrödinger’s original formulation connects the problem to large deviation theory and quantum mechanics \cite{berezin2012schrodinger}, interpreting it through the lens of rare events in statistical mechanics. Via Sanov’s theorem \cite{sanov1958probability} and Gibbs’ principle \cite{dembo2009large}, the SB problem aligns with the maximum entropy principle, ensuring a solid probabilistic foundation. In contemporary machine learning, the SB serves as a unifying framework connecting OT \cite{shi2023diffusion}, stochastic control \cite{chen2021stochastic}, and generative modeling \cite{de2021diffusion}. This perspective has yielded efficient algorithms for distribution interpolation, domain adaptation, and constrained generation, exhibiting both theoretical elegance and practical utility.

Beyond the Schrödinger bridge, stochastic interpolation has emerged as a broader framework for constructing probabilistic paths between prescribed endpoint distributions. In contrast to deterministic interpolation, stochastic interpolation incorporates random fluctuations into the intermediate dynamics, making it particularly suitable for modeling uncertainty, diffusion, and multi-modal transport behavior. Recent works have shown that this viewpoint provides a unifying perspective for flows, diffusions, and distribution matching, and has become increasingly relevant in generative modeling and probabilistic transport. From the perspective of related work, stochastic interpolation therefore serves as a natural extension of entropy regularized and bridge type formulations, while offering additional flexibility in the design of intermediate stochastic dynamics.

\subsection{Schrödinger bridge}
The most probable stochastic evolution between two prescribed marginal distributions $\rho_0$ and $\rho_1$ corresponds to the solution of the Schrödinger bridge problem:
\begin{equation}
    \mathcal{P}_{SB} := \mbox{argmin} \{D_{KL}(P\|W)|P\in \mathcal{D}(\rho_0,\rho_1)\},
\end{equation}
where $\mathcal{D}$ is the space of probability measures on $\Omega = C([0,1],\mathcal{X})$, and $W\in M_+({\Omega})$ is the reference path measure on $\Omega$. Here $\mathcal{X}$ is a complete connected Riemannian manifold without boundary, $M_+({\Omega})$ denotes positive measures on $\Omega$.

The goal of Schr\"odinger bridge is to obtain a pair $(\rho,u)$ that solves
\begin{numcases}{}
    &$\min\limits_{u,\rho} \mathcal{J}(u)=\int_0^1 \int_{\mathbb{R}^d}\|u_t\|^2\rho(t,x)\mbox{d}x\mbox{d}t,$\\
    &subject to $\partial_t \rho\,+\nabla\cdot (u\rho)=\varepsilon\Delta \rho$, \\
    &\qquad\quad $\rho(0)= \rho_0, \rho(1)=\rho_1.$
\end{numcases}

\subsection{Stochastic interpolation}
The stochastic interpolation~\cite{JMLR:v26:23-1605} between two prescribed marginal distributions \(\rho_0\) and \(\rho_1\)
is described by a probability measure \(\mathbb P\) on the path space $
\Omega := C([0,1],\mathcal X)$,
such that the associated canonical process \((X_t)_{t\in[0,1]}\) satisfies
\begin{equation*}
X_0 \sim \rho_0,
\qquad
X_1 \sim \rho_1.
\end{equation*}
Here \(\mathcal X\) is a complete connected Riemannian manifold, and \(\Omega\) denotes the space of
continuous trajectories on \([0,1]\).
A stochastic interpolation is typically realized as the law of a diffusion process solving
\begin{equation*}
\mathrm{d}X_t = b(t,X_t)\,\mathrm{d}t + \sigma(t,X_t)\,\mathrm{d}W_t,
\end{equation*}
where \(b\) is the drift term, \(\sigma\) denotes the noise intensity, and \(W_t\) is a standard Brownian motion.

The corresponding interpolating family of marginals is given by
\begin{equation*}
\rho_t := (X_t)_{\#}\mathbb P,
\qquad t\in[0,1],
\end{equation*}
so that \((\rho_t)_{t\in[0,1]}\) provides a probabilistic curve connecting \(\rho_0\) and \(\rho_1\).

Stochastic interpolation arises naturally in several settings, including stochastic control, diffusion processes, entropic optimal transport, and Schrödinger bridge problems. In all these cases, the central object is an evolving family of probability measures connecting prescribed endpoint distributions under random dynamics.

\section{Problem setup}
In this section, we formulate the stochastic interpolation problem at the continuum and particle levels, and state the formula of splitting detection objective, leaving numerical choices and implementation details to a later section.
\subsection{Particle approximation and stochastic interpolation}

Let \(\mu_0\) and \(\mu_1\) be two probability distributions on \(\mathbb R^d\), representing the initial and target distributions, respectively. We are interested in constructing a stochastic interpolation between \(\mu_0\) and \(\mu_1\), together with a velocity field that minimizes a kinetic type energy under noisy dynamics.

Our main objective is to formulate a mathematically tractable criterion for splitting detection in the interpolating dynamics, and in particular to detect the earliest splitting time and its associated spatial location.

Let \(\pi \in \Pi(\mu_0,\mu_1)\) be a prescribed coupling between the endpoint distributions. 
We sample i.i.d. endpoint pairs
\[
(X_0^{(i)},X_1^{(i)}) \sim \pi, \qquad i=1,\dots,N.
\]
In particular, the marginals satisfy $X_0^{(i)} \sim \mu_0,\, X_1^{(i)} \sim \mu_1$.

For a generic coupled endpoint pair \((X_0,X_1)\sim \pi\), we introduce a stochastic interpolation ansatz of the form
\begin{equation}\label{eq:ps_interp}
X_t=\alpha(t)X_0+\beta(t)X_1+\gamma(t)Z,
\qquad t\in[0,1],
\end{equation}
where \(Z\) is an auxiliary random variable, and \(\alpha(t)\), \(\beta(t)\), and \(\gamma(t)\) are scalar interpolation coefficients satisfy $\alpha(0)=\beta(1)=1,\alpha(1)=\beta(0)=\gamma(0)=\gamma(1)=0,\gamma(t)>0$. This ansatz induces a family of approximate trajectories connecting the sampled endpoints.

Differentiating \eqref{eq:ps_interp} with respect to time yields the pathwise velocity
\begin{equation}
\dot X_t=\dot\alpha(t)X_0+\dot\beta(t)X_1+\dot\gamma(t)Z.
\label{eq:pathwise_velocity}
\end{equation}
Here \(\dot X_t\) is understood as the Nelson derivative associated with the stochastic process \(X_t\).
This quantity is, in general, not a deterministic function of the current state \(X_t\). We therefore denote the associated Eulerian velocity field by conditional averaging,
\begin{equation}
v(x,t):=\mathbb E[\dot X_t \mid X_t=x],
\qquad (x,t)\in \mathbb R^d\times[0,1].
\label{eq:eulerian_velocity}
\end{equation}

The stochastic interpolation problem can be formulated as
\begin{numcases}{}
    &$\min\limits_{v} \mathcal E(v)=\mathbb{E}\left[\int_0^1\|v(x,t)\|^2\mbox{d}t\right],$\\
    &$v(x,t):=\mathbb E[\dot X_t \mid X_t=x]$, \\
    & $X_t=\alpha(t)X_0+\beta(t)X_1+\gamma(t)Z$,\\  &$\alpha(0)=\beta(1)=1,\alpha(1)=\beta(0)=\gamma(0)=\gamma(1)=0,\gamma(t)>0$,\\
    &$X_0\sim \mu_0(x), X_1\sim \mu_1(y),Z\sim \mathcal{N}(0,I_d)$
\end{numcases}

We next introduce the variational dynamics associated with the induced Eulerian flow. To motivate the second order viewpoint, we briefly recall the classical Jacobi field in Riemannian geometry.

\subsection{Jacobi field and variational dynamics}
In the classical Riemannian setting, let $(\mathcal M,g)$ be a smooth Riemannian manifold, and let $\Gamma:[0,T]\to \mathcal M$ be a geodesic. A vector field $
J(t)\in T_{\Gamma(t)}\mathcal M$
along $\Gamma$ is called a \emph{Jacobi field} if it satisfies
\begin{equation}
\frac{D^2}{dt^2}J(t)+R\bigl(J(t),\dot \Gamma(t)\bigr)\dot \Gamma(t)=0,
\qquad t\in[0,T],
\label{eq:Jacobi_classical}
\end{equation}
where $\frac{D}{dt}$ denotes the covariant derivative along $\Gamma$, and $R$ is the Riemann curvature tensor.

\begin{remark}[Standard meaning of the term Jacobi field]
Equation \eqref{eq:Jacobi_classical} is a second order linear equation along the reference geodesic $\Gamma$. 
Accordingly, a Jacobi field is determined by the pair of initial data $
J(0),\, \frac{D}{dt}J(0)$.
Thus, in the standard differential geometric terminology, the notion of a Jacobi field is intrinsically second order.
\end{remark}

This classical second order equation provides the geometric prototype for the deviation dynamics of nearby trajectories. In our setting, the corresponding first order variational equation is induced by the flow map, and the associated second order equation can be derived by differentiating once more in time.
\begin{proposition}
Let $v:[0,T]\times \mathbb R^d\to \mathbb R^d$ be a sufficiently smooth velocity field, and let
\[
\frac{\mathrm{d}}{\mathrm{d}t}x(t)=v\bigl(t,x(t)\bigr)
\]
be the associated flow equation. Consider a smooth one parameter family of trajectories
$
X:(-\varepsilon,\varepsilon)\times[0,T]\to\mathbb R^d, \frac{\partial}{\partial t}X(s,t)=v\bigl(t,X(s,t)\bigr)$, here \(s\in(-\varepsilon,\varepsilon)\) is a variation parameter labeling a family of nearby trajectories, while \(t\) denotes the time variable along each trajectory.
Define the variational field along the reference trajectory $x(t):=X(0,t)$ by
\[
J(t):=\frac{\partial X}{\partial s}(0,t),\quad J(0):=\frac{\partial X}{\partial s}(0,t)\bigg|_{t=0}.
\]
Then Jacobi field \(J\) satisfies the first order variational equation
\begin{equation}
\frac{\mathrm{d}}{\mathrm{d}t}J(t)=\nabla_x v\bigl(t,x(t)\bigr)\,J(t).
\label{eq:th_jacobian}
\end{equation}
Equivalently, it equals the second order Jacobi field
\begin{equation}
\frac{d^2}{dt^2}J(t)+M(t)J(t)=0.
\label{eq:second_variation}
\end{equation}
where
\begin{equation}
M(t):=
-
\Bigl(
\partial_t \nabla_x v
+
(v\cdot \nabla_x)(\nabla_x v)
+
(\nabla_x v)^2
\Bigr)\bigl(t,x(t)\bigr),
\label{eq:dynamic_curvature_matrix}
\end{equation}
\end{proposition}

\begin{proof}
Differentiating the flow equation
\[
\frac{\partial}{\partial t}X(s,t)=v\bigl(t,X(s,t)\bigr)
\]
with respect to \(s\), and then evaluating at \(s=0\), we obtain
\[
\frac{\mathrm{d}}{\mathrm{d}t}J(t)
=
\nabla_x v\bigl(t,x(t)\bigr)\,J(t),
\]
which proves \eqref{eq:th_jacobian}. Differentiating once more in time gives
\[
\frac{\mathrm{d}^2}{\mathrm{d}t^2}J(t)
=
\frac{\mathrm{d}}{\mathrm{d}t}\Bigl(\nabla_x v\bigl(t,x(t)\bigr)\Bigr)J(t)
+
\nabla_x v\bigl(t,x(t)\bigr)\frac{\mathrm{d}}{\mathrm{d}t}J(t).
\]
Using the chain rule along the trajectory \(x(t)\),
\[
\frac{\mathrm{d}}{\mathrm{d}t}\Bigl(\nabla_x v\bigl(t,x(t)\bigr)\Bigr)
=
\Bigl(
\partial_t \nabla_x v
+
(v\cdot \nabla_x)(\nabla_x v)
\Bigr)\bigl(t,x(t)\bigr),
\]
Equation \eqref{eq:second_variation} follows immediately from the definition of \(M(t)\).
\end{proof}

\begin{remark}[Relation between the first and second order formulations]
For a general first order dynamical system, \eqref{eq:th_jacobian} is the fundamental linearization of the flow map and describes the infinitesimal separation of nearby trajectories. 
Equation \eqref{eq:second_variation} is not an independent definition, but is derived from \eqref{eq:th_jacobian} by differentiating once more in time. 
It describes the infinitesimal relative acceleration between nearby trajectories through the operator \(M(t)\).

In contrast, in Riemannian geometry the standard Jacobi equation \eqref{eq:Jacobi_classical} is itself the primary formulation. 
When the underlying flow is a geodesic flow, the operator \(M(t)\) is exactly the curvature operator
\[
J\mapsto R\bigl(J,\dot \Gamma(t)\bigr)\dot \Gamma(t),
\]
and \eqref{eq:second_variation} reduces to the classical Jacobi equation. 
Therefore, the first order variational equation and the second order Jacobi equation represent the same deviation mechanism viewed at two different levels: the former at the level of infinitesimal separation, and the latter at the level of infinitesimal relative acceleration.
\end{remark}

We now turn to the splitting detection problem for the interpolating particle system. 
The goal is to identify the earliest time at which the local trajectory configuration ceases to remain coherent and begins to separate, together with the corresponding spatial location.

We then study its spatial Jacobian matrix
\begin{equation}
\nabla_x v(x,t)\in \mathbb R^{d\times d}.
\label{eq:velocity_jacobian}
\end{equation}

Let \(x_t\) be a reference trajectory generated by the induced velocity field, namely,
\begin{equation}
\dot x_t = v(x_t,t).
\label{eq:reference_trajectory}
\end{equation}
Along this trajectory, we define the associated variational matrix \(J_t\in\mathbb R^{d\times d}\) as the fundamental matrix solution of the linearized equation
\begin{equation}
\dot J_t = \nabla_x v(x_t,t)\,J_t,
\qquad J_0=I_d.
\label{eq:variational_matrix}
\end{equation}
This matrix describes the infinitesimal deformation of nearby trajectories relative to the reference path \(x_t\).

To prepare for the splitting criterion developed in the next subsection, we introduce a local diagnostic based on the induced velocity gradient along the reference trajectory. Specifically, we define the symmetric part of the Jacobian matrix by
\begin{equation}
S(t):=\frac{\nabla_x v(x_t,t)+\nabla_x v(x_t,t)^\top}{2}\in\mathbb R^{d\times d}.
\label{eq:strain_tensor}
\end{equation}
The matrix \(S(t)\) is the rate of strain tensor associated with the induced flow along \(x_t\). Its spectrum characterizes the instantaneous local stretching and compression generated by the Eulerian velocity field. In particular, pronounced positive stretching directions provide a natural signal for the onset of local trajectory separation. For this reason, \(S(t)\) will serve as a basic ingredient in our subsequent splitting detection criterion.

\begin{problem}
\label{prob:core_problem}
Assume that an optimal interpolating flow has already been constructed, and let \(v:\mathbb R^d\times[0,1]\to\mathbb R^d\) be its induced Eulerian velocity field. The associated variational field $J$ satisfies
\[
\frac{\mathrm{d}}{\mathrm{d}t}J(t)=\nabla_x v\bigl(t,x(t)\bigr)\,J(t).\]
Denote the largest eigenvalue of symmetric part of the local velocity gradient $S(t)$ as
\[
\lambda_{\max}(t):=\lambda_{\max}\bigl(S(t)\bigr),
\]
where $S(t)$ is defines in Eq.(\ref{eq:strain_tensor}).
The core problem is to determine whether the onset of local path splitting along the interpolating flow can be characterized through the evolution of \(\lambda_{\max}(t)\), and in particular to define a rigorous and computable criterion that identifies the first splitting time $t_s\in(0,1)$ and the corresponding spatial location $x_{t_s}\in\mathbb R^d$ from the behavior of \(\lambda_{\max}(t)\).
\end{problem}

\subsection{Induced velocity field and Jacobian dynamics}

We now derive the Eulerian velocity field induced by the interpolation ansatz \eqref{eq:ps_interp} and compute its spatial Jacobian matrix. We restrict attention to the case of \(\mathbb R^2\).

\begin{proposition}
\label{prop:nabla_v_explicit}
Let $X_t=\alpha(t)X_0+\beta(t)X_1+\gamma(t)Z, t\in[0,1]$, \(X_0\sim \mathcal N(0,I_2)\), \(Z\sim \mathcal N(0,I_2)\), and \(X_0\), \(Z\) are mutually independent. Assume that the terminal random variable \(X_1\) is supported on the finite set $\{y_1,\dots,y_N\}\subset \mathbb R^2$,
with prior weights $\mathbb P(X_1=y_i)=\pi_i, \pi_i>0, \sum_{i=1}^N \pi_i=1$.
Then the spatial Jacobian matrix of \(v\) is given by
\begin{equation}
\nabla_x v(t,x)
=
A(t) I_2
+
\bigl(\beta'(t)-A(t)\beta(t)\bigr)\,
\frac{\beta(t)}{\sigma^2(t)}\,
\mathrm{Cov}_{w(t,x)}(Y),
\label{eq:explicit_nabla_v}
\end{equation}
where $A(t):=\frac{\alpha(t)\alpha'(t)+\gamma(t)\gamma'(t)}{\sigma^2(t)},\sigma^2(t):=\alpha^2(t)+\gamma^2(t)$, and \(Y\) denotes the discrete random variable taking values in \(\{y_1,\dots,y_N\}\) with posterior weights $
\mathbb P(Y=y_i)=w_i(t,x),\,i=1,\dots,N$,
where $
w_i(t,x)
={\pi_i\exp\!\left(-\frac{\|x-\beta(t)y_i\|^2}{2\sigma^2(t)}\right)}\Big/
{\sum_{j=1}^N \pi_j\exp\!\left(-\frac{\|x-\beta(t)y_j\|^2}{2\sigma^2(t)}\right)}$. The covariance $=
\mathrm{Cov}_{w(t,x)}(Y)
=
\sum_{i=1}^N w_i(t,x)\,y_i y_i^\top
-
m(t,x)m(t,x)^\top$
 with
$
m(t,x):=\sum_{i=1}^N w_i(t,x)\,y_i.$
\end{proposition}
\begin{proof}
    Conditioned on \(X_1=y_i\), we have
\begin{equation}\label{eq:th_conditional_law}
X_t\mid (X_1=y_i)\sim \mathcal N\!\bigl(\beta(t)y_i,\sigma^2(t)I_2\bigr).
\end{equation}
By Bayesian theorem, the posterior weights are
\begin{equation}\label{eq:th_weights}
w_i(t,x)
=
\mathbb P(X_1=y_i\mid X_t=x)
=
\frac{\pi_i\exp\!\left(-\frac{\|x-\beta(t)y_i\|^2}{2\sigma^2(t)}\right)}
{\sum_{j=1}^{N}\pi_j\exp\!\left(-\frac{\|x-\beta(t)y_j\|^2}{2\sigma^2(t)}\right)}.
\end{equation}
The posterior mean of the terminal point is
\begin{equation}\label{eq:th_m}
m(t,x):=\mathbb E[X_1\mid X_t=x]
=
\sum_{i=1}^{N}w_i(t,x)y_i.
\end{equation}

Define $U_t=\alpha(t)X_0+\gamma(t)Z$, we have $U_t\sim \mathcal{N}(0,(\alpha^2(t)+\gamma^2(t))I_2)$ and $\dot{U}_t=\alpha'(t)X_0+\gamma'(t)Z$. Since $\mathbb{E}[\dot{U}]=0,\mathbb{E}[{U}]=0$, the covariance
\begin{align*}
    \mathrm{Cov}(\dot{U},U)&=\mathbb{E}[\dot{U}U]=\mathbb{E}[(\alpha'(t)X_0+\gamma'(t)Z)(\alpha(t)X_0+\gamma(t)Z)]\\
    &=\alpha'(t)\alpha(t)\,\mathbb E[X_0X_0^\top]+\alpha'(t)\gamma(t)\,\mathbb E[X_0Z^\top]+\gamma'(t)\alpha(t)\,\mathbb E[ZX_0^\top]+\gamma'(t)\gamma(t)\,\mathbb E[ZZ^\top],
\end{align*}
and the covariance $\mathrm{Cov}(U_t)=(\alpha^2(t)+\gamma^2(t))I_2$
From the independence of $X_0$ and $Z$, we have $\mathbb E[X_0Z^\top]=\mathbb E[ZX_0^\top]=\mathbb E[X_0]\mathbb E[Z]=0$. The yields
\begin{equation*}
    \mathrm{Cov}(\dot{U},U)=\alpha'(t)\alpha(t)\,\mathbb E[X_0X_0^\top]+\gamma'(t)\gamma(t)\,\mathbb E[ZZ^\top]=(\alpha'(t)\alpha(t)
    +\gamma'(t)\gamma(t))I_2
\end{equation*}
By the standard Gaussian regression formula, we get
\begin{equation*}
\mathbb E[\dot{U}_t\mid U_t=u]
=
\operatorname{Cov}(\dot{U}_T,U_t)\,(\operatorname{Cov}U_t)^{-1}u=\frac{\alpha'(t)\alpha(t)
    +\gamma'(t)\gamma(t)}{\alpha^2(t)+\gamma^2(t)}u:=A(t)u,
\end{equation*}
where $A(t)=\frac{\alpha'(t)\alpha(t)
    +\gamma'(t)\gamma(t)}{\alpha^2(t)+\gamma^2(t)}$.
That means
\begin{align*}
    \mathbb{E}[\dot{X}_t\mid X_t=x,X_1=y_i]
=\mathbb{E}[\dot{U}_t+\beta'_tX_1 \mid X_t=x,X_1=y_i]=
A(t)\bigl(x-\beta(t)y_i\bigr)+\beta'(t)y_i.
\end{align*}
From Eq.(\ref{eq:th_m}) and above equation, we obtain
\begin{align*}
    v(t,x)=\mathbb{E}[\dot{X}_t\mid X_t=x]=A(t)\bigl(x-\beta(t)m(t,x)\bigr)+\beta'(t)m(t,x).
\end{align*}
Take gradient to $x$ yielding
\begin{align}\label{eq_gradv}
    \nabla_x v(t,x)=A(t)I_2+\bigl(\beta'(t)-A(t)\beta(t)\bigr)\nabla m.
\end{align}

Let
\[
s_i(t,x):=\log \pi_i-\frac{\|x-\beta(t)y_i\|^2}{2\sigma^2(t)},
\qquad
w_i(t,x)=\frac{e^{s_i(t,x)}}{\sum_{j=1}^N e^{s_j(t,x)}}.
\]
Then
\[
\nabla_x s_i(t,x)=-\frac{x-\beta(t)y_i}{\sigma^2(t)}.
\]
By the softmax gradient formula,
\[
\nabla_x w_i(t,x)
=
w_i(t,x)\left(\nabla_x s_i(t,x)-\sum_{j=1}^N w_j(t,x)\nabla_x s_j(t,x)\right).
\]
Using
\[
m(t,x)=\sum_{j=1}^N w_j(t,x)y_j,
\]
we obtain
\[
\sum_{j=1}^N w_j(t,x)\nabla_x s_j(t,x)
=
-\frac{x-\beta(t)m(t,x)}{\sigma^2(t)}.
\]
Hence
\[
\nabla_x w_i(t,x)
=
\frac{\beta(t)}{\sigma^2(t)}\,w_i(t,x)\bigl(y_i-m(t,x)\bigr).
\]
Therefore,
\[
\nabla_x m(t,x)
=
\sum_{i=1}^N y_i\otimes \nabla_x w_i(t,x)
=
\frac{\beta(t)}{\sigma^2(t)}
\sum_{i=1}^N w_i(t,x)\,y_i\bigl(y_i-m(t,x)\bigr)^\top.
\]
Expanding the last term gives
\[
\nabla_x m(t,x)
=
\frac{\beta(t)}{\sigma^2(t)}
\left(
\sum_{i=1}^N w_i(t,x)\,y_i y_i^\top
-
m(t,x)m(t,x)^\top
\right),
\]
which is exactly
\[
\nabla_x m(t,x)
=
\frac{\beta(t)}{\sigma^2(t)}\,\mathrm{Cov}_{w(t,x)}(Y).
\]
Then Eq.(\ref{eq_gradv}) becomes
\begin{equation*}
    \nabla_x v(t,x)
=
A(t) I_2
+
\bigl(\beta'(t)-A(t)\beta(t)\bigr)\,
\frac{\beta(t)}{\sigma^2(t)}\,
\mathrm{Cov}_{w(t,x)}(Y).
\end{equation*}
\end{proof}

The next proposition shows that the largest eigenvalue of the symmetric part of the Jacobian characterizes the strongest instantaneous stretching direction of the local flow, and therefore provides a natural dynamical indicator for path splitting.
\begin{proposition}\label{prop:max_eig_stretching}
Assume the velocity field \(v(t,x)\) is \(C^1\) in \(x\).
Let $\lambda_{\max}(t):=\lambda_{\max}(S(t))$, then for every \(t\in(0,1)\), $
\lambda_{\max}(t)
=
\max_{\|e\|=1} e^\top S(t)e.$
In particular, if \(\lambda_{\max}(t)>0\), then there exists a unit direction \(e_t\) such that an infinitesimal perturbation \(\delta x_t=\varepsilon e_t\) satisfies
\[
\frac{\mathrm{d}}{\mathrm{d}t}\frac12\|\delta x_t\|^2
=
\delta x_t^\top S(t)\delta x_t
=
\lambda_{\max}(t)\|\delta x_t\|^2
>0.
\]
Hence the local flow is instantaneously stretching along the principal eigenvector direction associated with \(\lambda_{\max}(t)\). Consequently, large positive values of \(\lambda_{\max}(t)\) provide a natural dynamical indicator of local stretching, and therefore a candidate signature of path splitting.
\end{proposition}

\section{Algorithm}

\subsection{Discontinuous Trajectory Modeling via neural networks}
To characterize the non-smooth evolution of data distributions, we propose a trajectory model governed by discontinuous neural networks. We define the state transition from the initial configuration $\mathbf{x}_0$ to the target configuration $\mathbf{x}_1 $ using time-dependent interpolation coefficients.
\begin{equation}
    X_t=\alpha(t)X_0+\beta(t)X_1+\gamma(t)Z,
\qquad t\in[0,1]
\end{equation}
Unlike standard linear interpolation, our approach employs a Discontinuous Dense Layer to learn the coefficients $\alpha(t), \beta(t), \gamma (t)$. This architecture utilizes a modified Heaviside activation function to capture abrupt transitions in the path, allowing the model to represent non-adiabatic topological changes.

The instantaneous dynamics of the system are represented by the velocity field $\mathbf{v}_t(x)$. Since the analytical derivative of the discontinuous PathNet may be ill-defined at jump points, we utilize the Central Difference Method to robustly estimate the velocity:
\begin{equation}
    V(t)=\dot{X}(t)=\dot\alpha(t)X_0+\dot\beta(t)X_1+\dot\gamma(t)Z
\end{equation}
\begin{align}
\dot{\alpha}_{\phi}(t_i) &\approx \frac{\alpha_{\phi}(t_{i+1}) - \alpha_{\phi}(t_{i-1})}{2\Delta t} \ , \\
\dot{\beta}_{\psi}(t_i) &\approx \frac{\beta_{\psi}(t_{i+1}) - \beta_{\psi}(t_{i-1})}{2\Delta t} \ ,\\
\dot{\gamma}_{\theta}(t_i) &\approx \frac{\gamma_{\theta}(t_{i+1}) - \gamma_{\theta}(t_{i-1})}{2\Delta t}.
\end{align}
This numerical scheme ensures stable gradient estimation across the temporal domain, providing the necessary kinematic information for subsequent physical analysis.

\subsection{Loss Function Designed for Physical Constraints}
To ensure the learned trajectories are physically meaningful and adhere to the predefined boundary conditions, we formulate a multi-objective loss function. The total loss $\mathcal{L}_{total}$ is composed of a boundary constraint term and a regularity term:
\subsubsection{Boundary Constraint Loss}
Since the system must transition precisely from the initial state $\mathbf{x}_0$ at $t=0$ to the target state $\mathbf{x}_1$ at $t=1$, we impose a boundary loss on the interpolation coefficients $a(t)$ and $b(t)$. This ensures that the network reconstructs the identity mapping at the temporal boundaries:
\begin{equation}
    \mathcal{L}_{bound}=(\alpha(0)-1)^2 + \beta(0)^2 + \alpha(1)^2 + (\beta(1)-1)^2 + \gamma(0)^2 + \gamma(1)^2.
\end{equation}

\subsubsection{Energy Regularization}
To prevent the model from learning erratic or non-physical paths, we introduce a kinetic energy penalty. This term acts as a regularizer that minimizes the $L_2$ norm of the velocity field across the entire temporal domain $[0, 1]$. By minimizing the mean square of the velocity, we encourage the model to find the principle of lowest energy path:
\begin{equation}
    \mathcal{L}_{energy}=\int_0^1 \mathbb{E}[\|\mathbf{v}_t(x)\|^2] dt=\int_0^1\int_X \rho_t(x)\|\mathbf{v}_t(x)\|^2 dxdt.
\end{equation}

\subsubsection{Total Objective and Optimization}
The final optimization objective is a weighted combination of these terms, where $\lambda_{bound}$ and $\lambda_{energy}$ are hyperparameters that balance the strictness of the boundary conditions against the smoothness of the trajectory:$$\mathcal{L}_{total} = \lambda_{bound} \mathcal{L}_{bound} + \lambda_{energy} \mathcal{L}_{energy}$$We employ the AdamW optimizer with decoupled weight decay to minimize this objective, allowing the discontinuous PathNet to autonomously discover the critical time points of topological change while maintaining global stability.



\subsection{Statistical Indicators and Top-K Analysis Framework}
To validate the sensitivity of the proposed model to localized topological changes, we recognize that splitting is a localized phenomenon rather than a global shift,so we implement a Top-$K$ Selection Strategy. Building upon this, we construct a tripartite indicator system to serve as a robust early-warning signal for topological transitions. 
These algorithm can be written as follows:



\begin{algorithm}[H]
\caption{Physics-Informed Path Network Training and Indicator Analysis}
\label{alg:path_net}
\begin{algorithmic}[1]
\REQUIRE Source samples $X_0 \sim \mathcal{N}(0, I)$, Target samples $X_1 \sim \mathbb{P}_{\text{data}}$, Neural Network parameterized by $\theta$.
\REQUIRE Time steps $N$, batch size $B$, learning rate $\eta$.
\STATE \textbf{Phase 1: Training the Path Network}
\WHILE{not converged}
    \STATE Sample batch $x_0 \sim X_0$, $x_1 \sim X_1$, and noise $z \sim \mathcal{N}(0, I)$.
    \STATE Sample $t \sim \mathcal{U}(0, 1)$.
    \STATE Predict coefficients $\alpha(t), \beta(t), \gamma(t) = \text{PathNet}_\theta(t)$.
    \STATE Compute the velocity field $v_t = \alpha'(t)x_0 + \beta'(t)x_1 + \gamma'(t)z$.
    \STATE Calculate boundary loss $\mathcal{L}_{\text{bound}}$ enforcing constraints at $t=0$ and $t=1$.
    \STATE Calculate energy loss $\mathcal{L}_{\text{energy}} = \int_0^1 \mathbb{E}[\|\mathbf{v}_t(x)\|^2] dt$.
    \STATE Compute total loss $\mathcal{L} = \lambda_1 \mathcal{L}_{\text{bound}} + \lambda_2 \mathcal{L}_{\text{energy}}$.
    \STATE Update $\theta \leftarrow \theta - \eta \nabla_\theta \mathcal{L}$.
\ENDWHILE
\STATE \textbf{Phase 2: Trajectory Generation and Indicator Extraction}
\FOR{$t \in [0, 1]$ with step $\Delta t$}
    \STATE Calculate interpolated particle positions $x_t = a(t)x_0 + b(t)x_1 + c(t)z$.
    \STATE Construct gradient of velocity field $ \nabla_x v_t$ with Eq.(\ref{eq:th_jacobian}) and compute the maximum eigenvalue $\lambda_{\max}(S(t))$, $S(t)$ is defined by Eq.(\ref{eq:strain_tensor}).
\ENDFOR
\RETURN Optimal trajectories $\{x_t\}_{t=0}^1$ and corresponding phase-transition indicators.
\end{algorithmic}
\end{algorithm}

\section{Experiments}
In the following experiments, we demonstrate the effectiveness of the proposed framework through several visualizations and quantitative measures. Specifically, we present a series of evolution maps that show the particle density, the $\delta$, and the distribution of the maximum eigenvalue. These maps help us observe how the transport geometry changes during the evolution process.

To further analyze the overall behavior, we also track the cumulative energy and eigenvalue integrals over time. By combining spatial visualizations with these global indicators, we can better understand the relationship between local curvature and the convergence of probability paths. This analysis also provides useful insights into the stability and sensitivity of the generative model when dealing with complex distributions.
\subsection{Two Moons}
The Two Moons dataset represents a typical case of binary mode separation. In this task, the initial Gaussian distribution needs to split into two arc-shaped structures. Our framework shows a clear spatiotemporal pattern during this process.

From the Grid Evolution maps, particle separation begins around $t \in [0.7, 0.8]$. At the same time, the Delta ($\delta$) and eigenvalue maps display strong intensity in the region where the flow starts to split. These high-intensity areas indicate that the transport dynamics are changing rapidly in this region.

At the global level, we track the Top-30 eigenvalue integral together with its slope over time. The eigenvalue integral reaches a clear peak at time $t^*$, while the slope shows a rapid change in the earlier interval. This behavior appears when the particle flow begins to split into two arcs. The consistency between the spatial patterns and these indicator changes suggests that the proposed metric can effectively identify the critical stage of mode separation in the Two Moons transport process.

\begin{figure}[H]
    \centering
    \includegraphics[width=12cm, height=5cm]{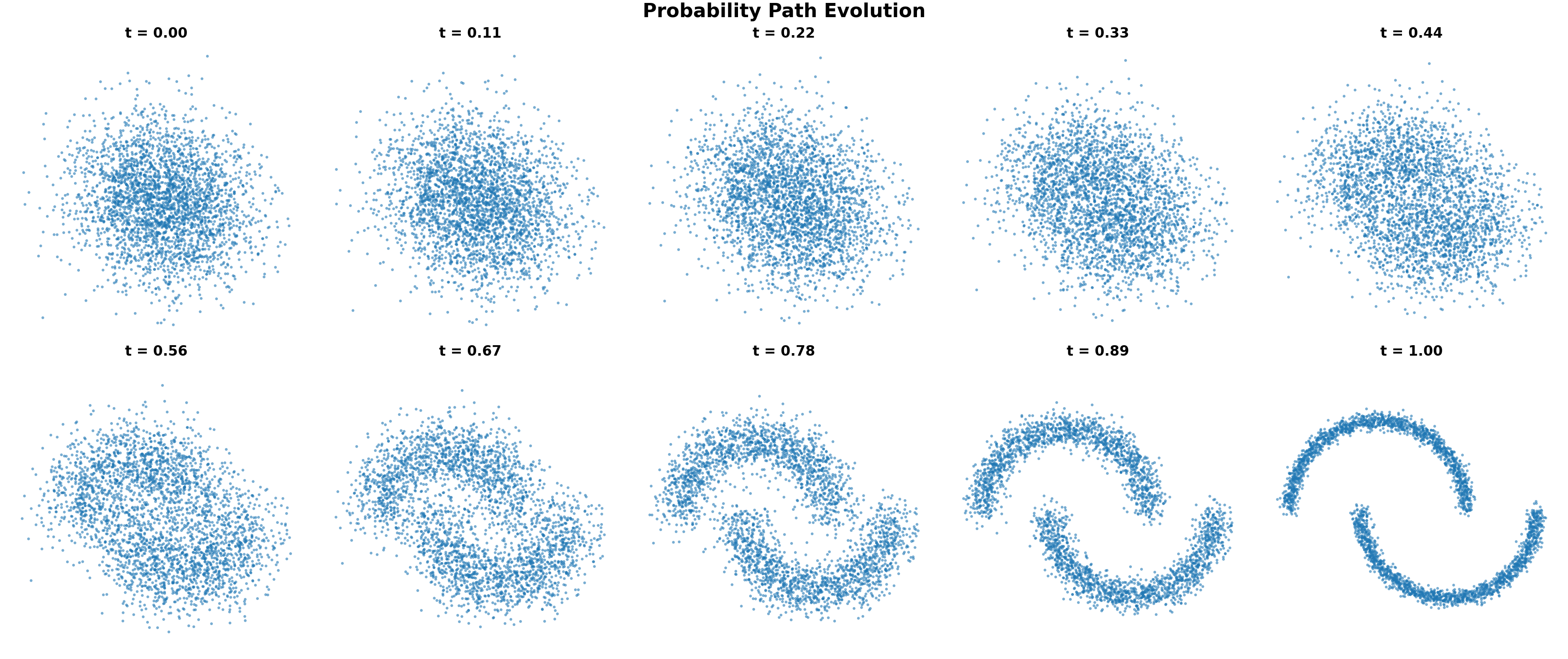}
    \caption{Probability Path Evolution. Snapshots of particle transport from a Gaussian prior to the Two Moons distribution. The $x, y$ axes represent 2D space, illustrating the structural splitting from $t=0.00$ to $t=1.00$.}
    \label{moons 3000}
\end{figure}

\begin{figure}[H]
    \centering
    \includegraphics[width=12cm, height=5cm]{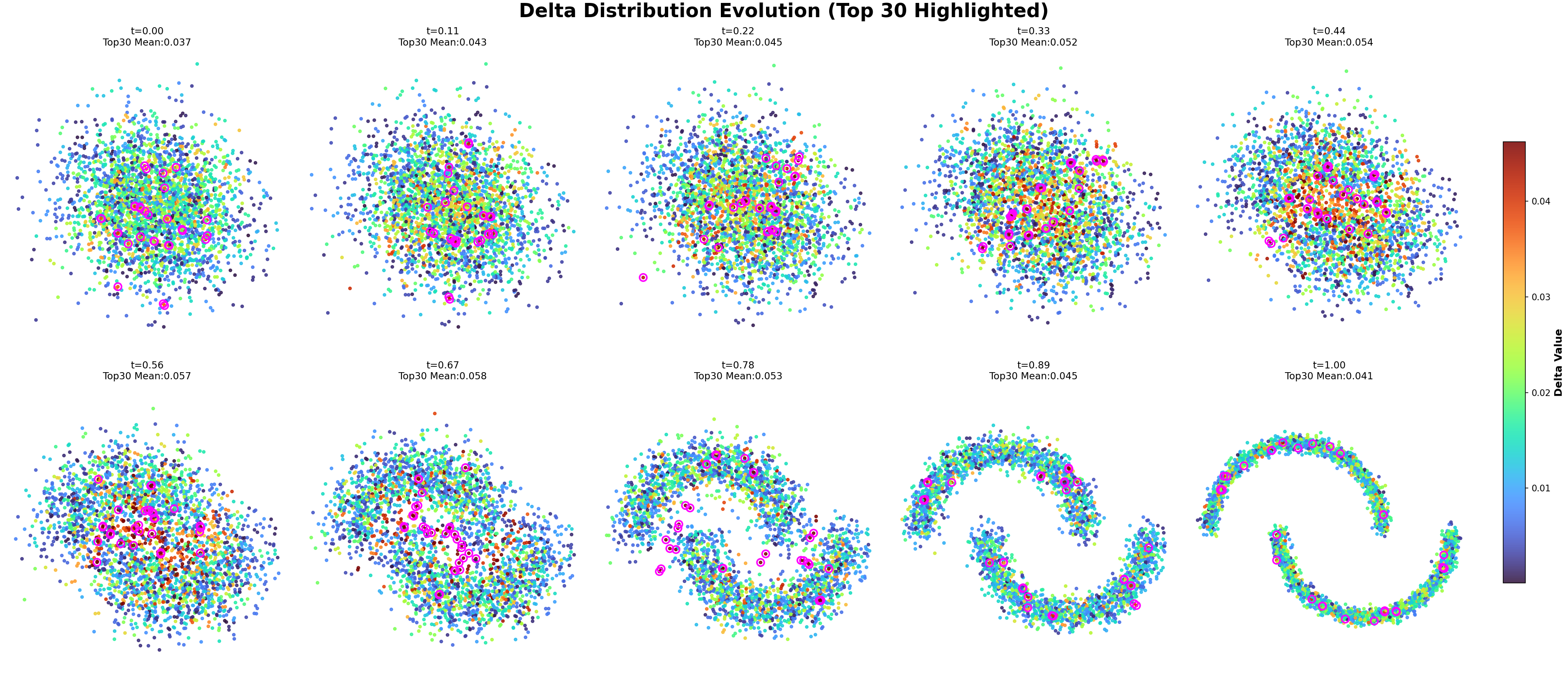}
    \caption{Delta ($\Delta$) Distribution Evolution. Points are colored by their $\Delta$ value (splitting intensity). The $x, y$ axes are spatial coordinates. Pink circles highlight the Top-30 particles concentrated at the flow's bifurcation neck.}
    \label{moons 3000 delta}
\end{figure}

\begin{figure}[H]
    \centering
    \includegraphics[width=12cm, height=5cm]{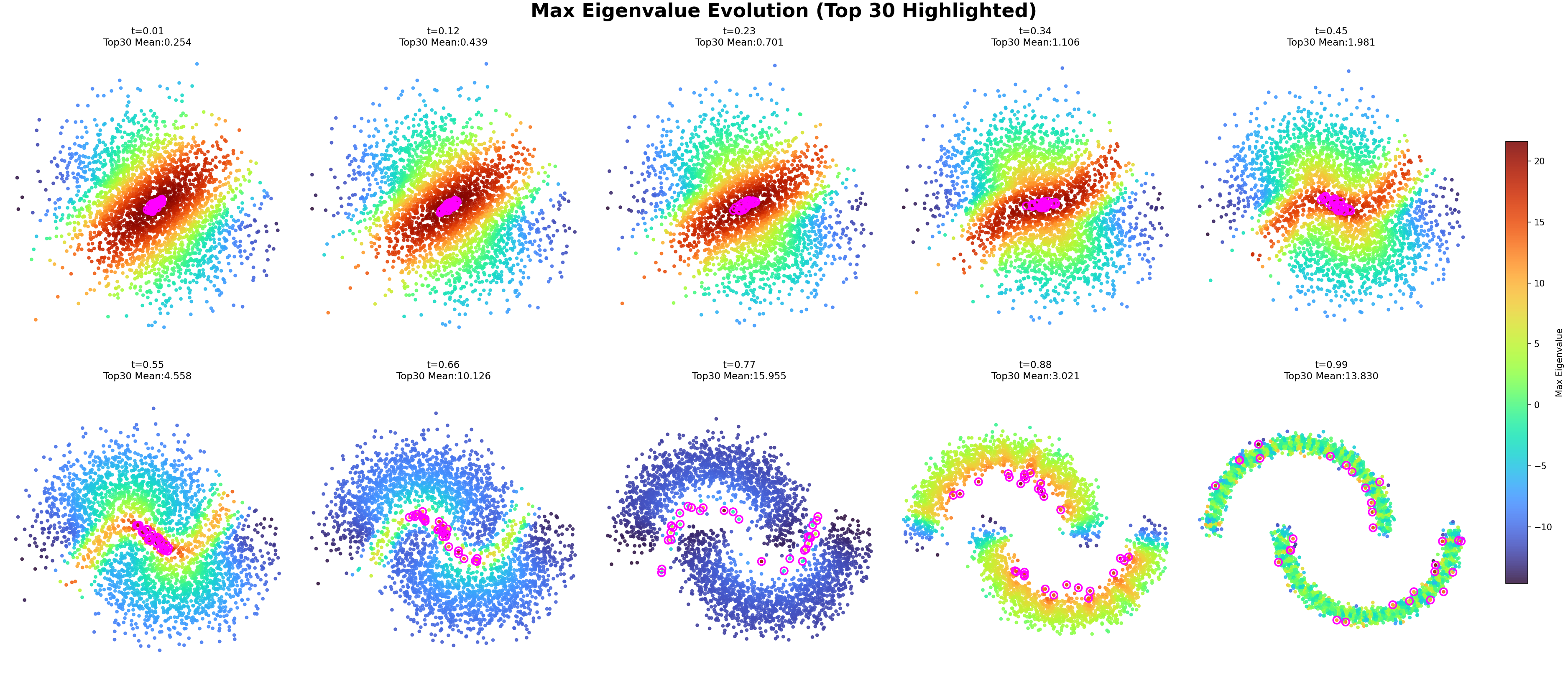}
    \caption{Max Eigenvalue Evolution. Points are colored by the Jacobian's max eigenvalue, indicating local manifold curvature. The $x, y$ axes are spatial coordinates, with pink circles tracking the Top-30 high-curvature hotspots.}
    \label{moons 3000 eig}
\end{figure}

\begin{figure}[H]
    \centering
    \includegraphics[width=12cm, height=6cm]{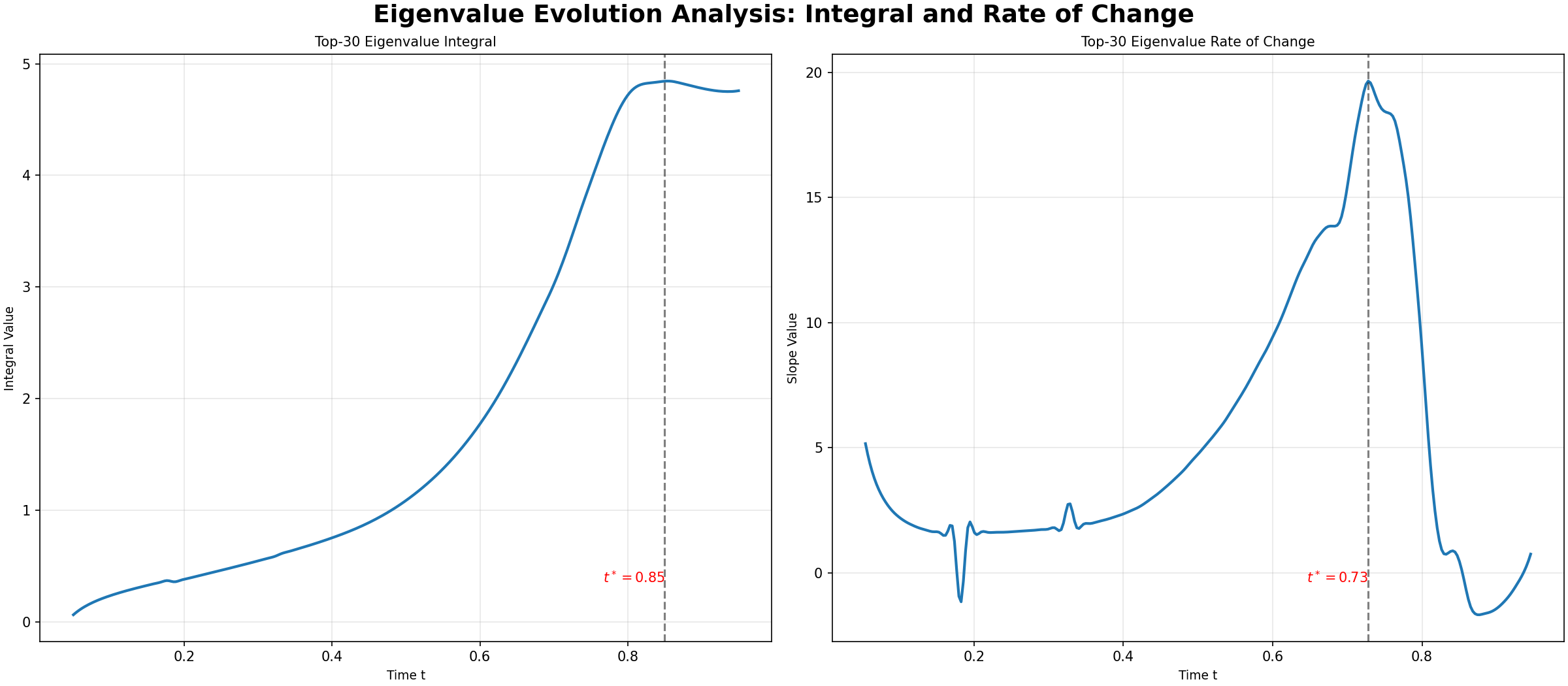}
    \caption{Eigenvalue Indicator Analysis. (Left) Cumulative integral of Top-30 eigenvalues. (Right) Rate of Top-30 eigenvalues' change (slope). The $x$-axis is time $t$; the vertical dashed line marks the peak $t^*$ signaling the onset of splitting.}
    \label{moons 3000 indicator}
\end{figure}

\subsection{Checkerboard}
The Checkerboard dataset requires the probability mass to be separated into a grid of disjoint square regions. During this process, the distribution undergoes repeated splitting and compression as it evolves toward the target pattern.

From the Grid Evolution maps, the particle distribution gradually forms a checkerboard structure. At the same time, the Delta ($\delta$) and eigenvalue maps reveal several high-intensity regions that appear near the boundaries between different modes. These regions form a grid-like pattern and indicate where the transport flow becomes more unstable.

From a global perspective, we examine the Top-30 eigenvalue integral and its slope during the evolution. The eigenvalue integral increases and reaches its peak around time $t^*$, while the slope changes sharply in the same region. This moment corresponds to the stage when the grid structure begins to separate into distinct square modes. The agreement between the spatial maps and these indicator variations suggests that the proposed metric can capture the key transition stage where the checkerboard structure first forms.

\begin{figure}[H]
    \centering
    \includegraphics[width=12cm, height=5cm]{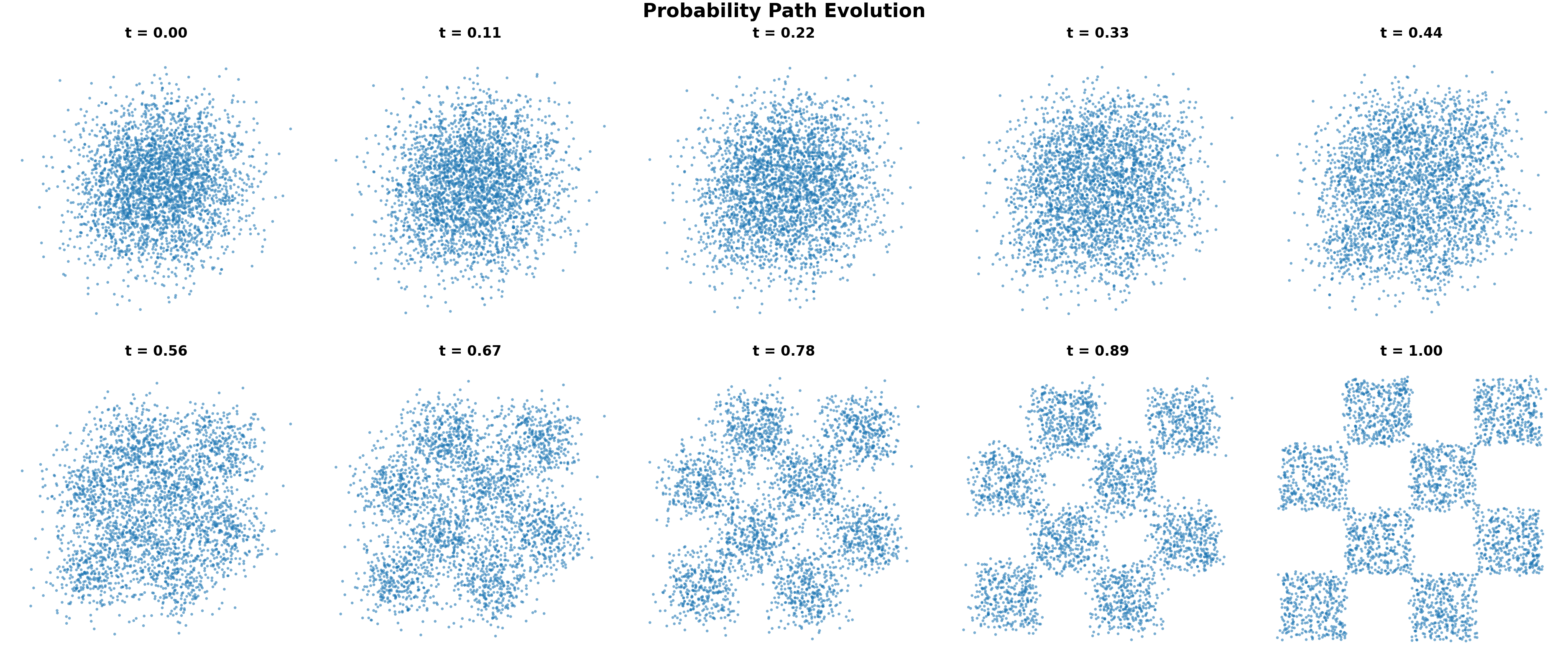}
    \caption{Probability Path Evolution. Snapshots of particle transport from a Gaussian prior to the Checkerboard distribution. The $x, y$ axes represent 2D space, illustrating the structural splitting from $t=0.00$ to $t=1.00$.}
    \label{checkboard 3000}
\end{figure}

\begin{figure}[H]
    \centering
    \includegraphics[width=12cm, height=5cm]{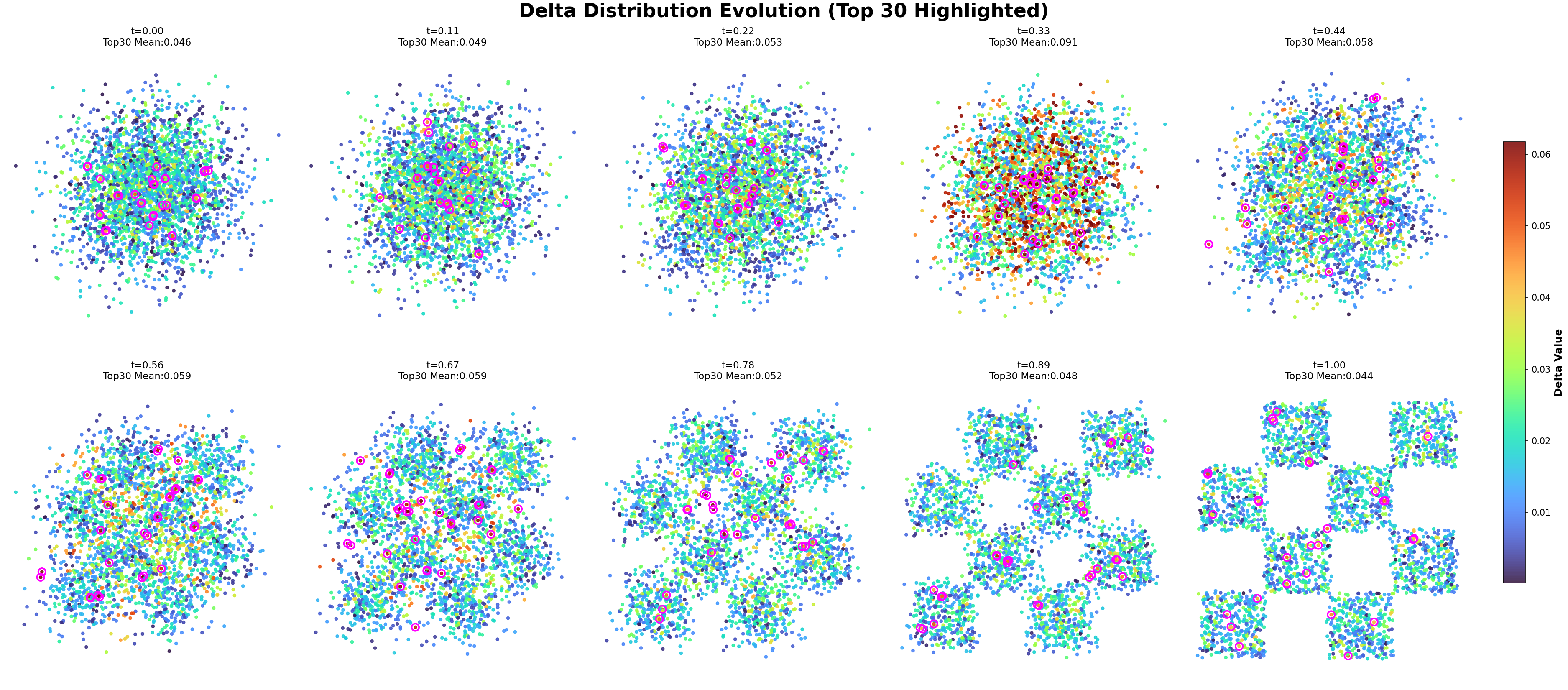}
    \caption{Delta ($\Delta$) Distribution Evolution. Points are colored by their $\Delta$ value (splitting intensity). The $x, y$ axes are spatial coordinates. Pink circles highlight the Top-30 particles concentrated at the flow's bifurcation neck.}
    \label{checkboard 3000 delta}
\end{figure}

\begin{figure}[H]
    \centering
    \includegraphics[width=12cm, height=5cm]{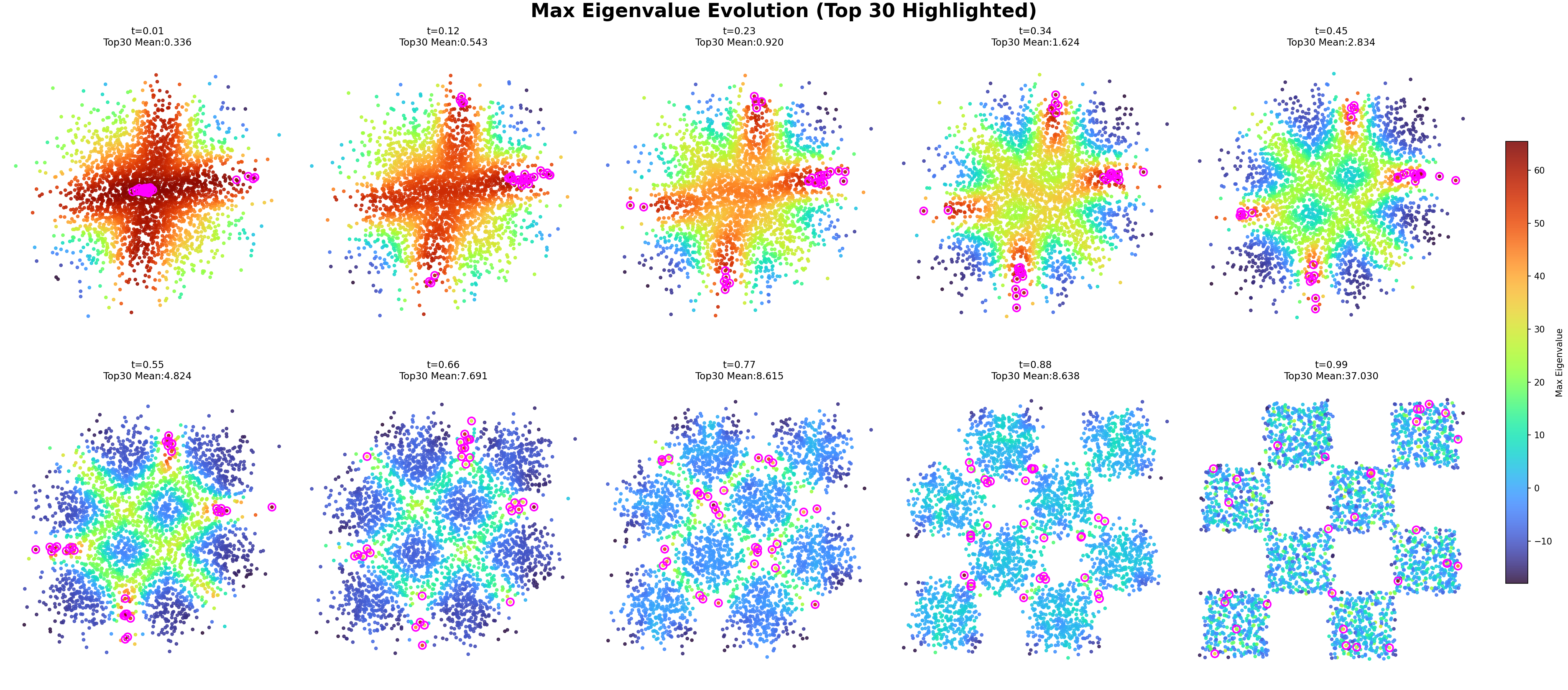}
    \caption{Max Eigenvalue Evolution. Points are colored by the Jacobian's max eigenvalue, indicating local manifold curvature. The $x, y$ axes are spatial coordinates, with pink circles tracking the Top-30 high-curvature hotspots.}
    \label{checkboard 3000 eig}
\end{figure}

\begin{figure}[H]
    \centering
    \includegraphics[width=12cm, height=6cm]{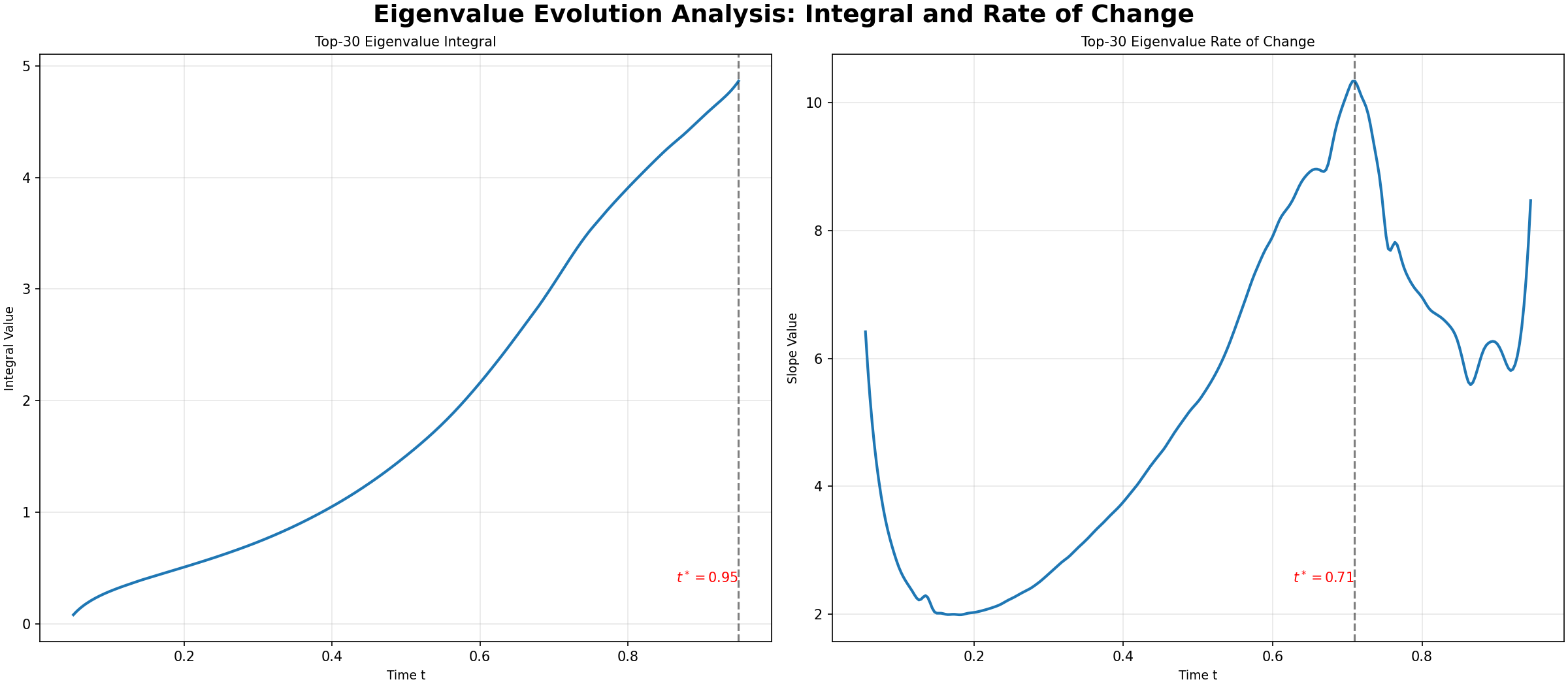}
    \caption{Eigenvalue Indicator Analysis. (Left) Cumulative integral of Top-30 eigenvalues. (Right) Rate of Top-30 eigenvalues' change (slope). The $x$-axis is time $t$; the vertical dashed line marks the peak $t^*$ signaling the onset of splitting.}
    \label{checkboard 3000 indicator}
\end{figure}

\subsection{S-Curve}
The S-Curve dataset represents probability transport along a continuous but highly non-linear manifold. During this process, the particle distribution gradually bends and stretches to match the characteristic S-shaped structure.

From the spatial heatmaps, the Top-30 particles tend to concentrate in regions with strong curvature. These regions correspond to the bending segments of the manifold and align well with the S shape observed in the Grid Evolution maps. The concentration of particles in these areas indicates where the transport dynamics change more rapidly.

At the global level, we analyze the Top-30 eigenvalue integral and its slope over time. The eigenvalue integral reaches its maximum near time $t^*$, and the slope exhibits a noticeable change around the same stage. This period corresponds to the moment when the manifold experiences the strongest bending and deformation. The consistency between the spatial heatmaps and these indicator changes suggests that the metric can effectively reflect the geometric curvature during the evolution.

\begin{figure}[H]
    \centering
    \includegraphics[width=12cm, height=5cm]{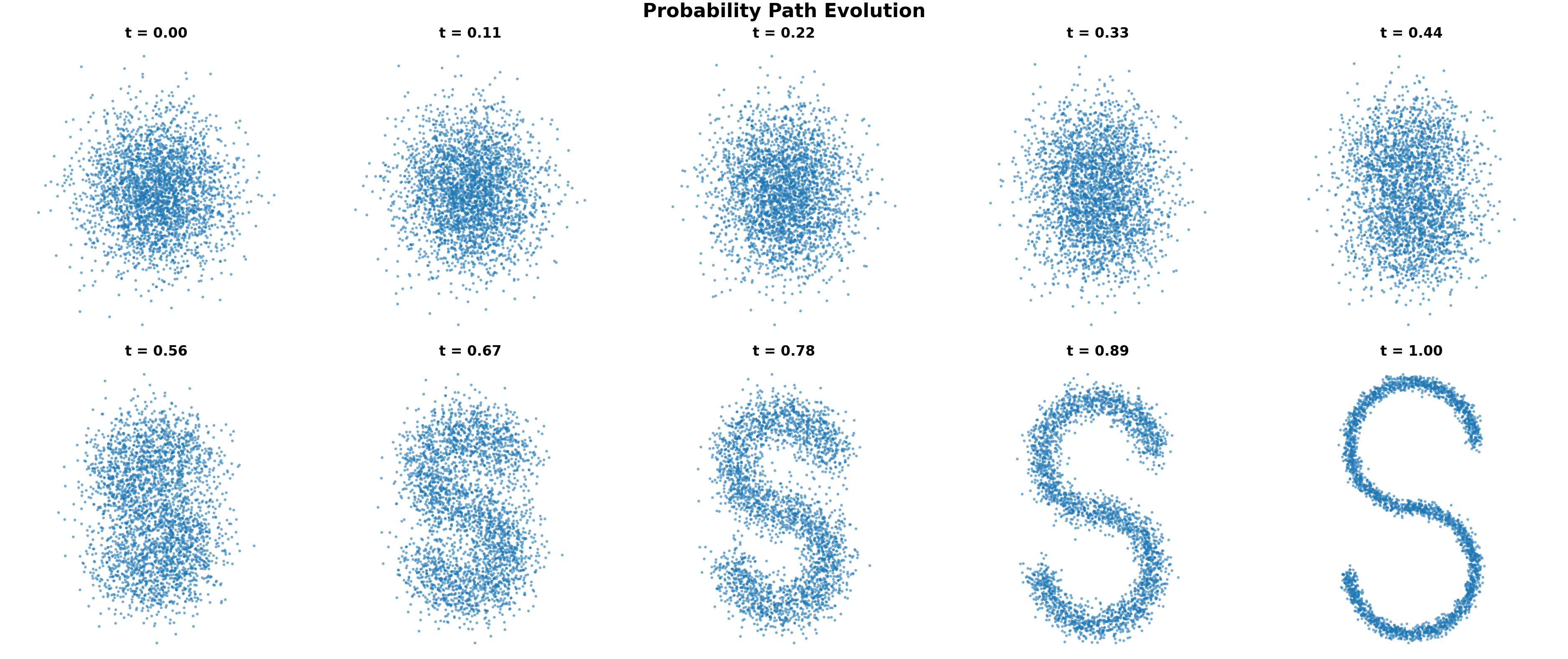}
    \caption{Probability Path Evolution. Snapshots of particle transport from a Gaussian prior to the S-Curve distribution. The $x, y$ axes represent 2D space, illustrating the structural splitting from $t=0.00$ to $t=1.00$.}
    \label{S 3000}
\end{figure}

\begin{figure}[H]
    \centering
    \includegraphics[width=12cm, height=5cm]{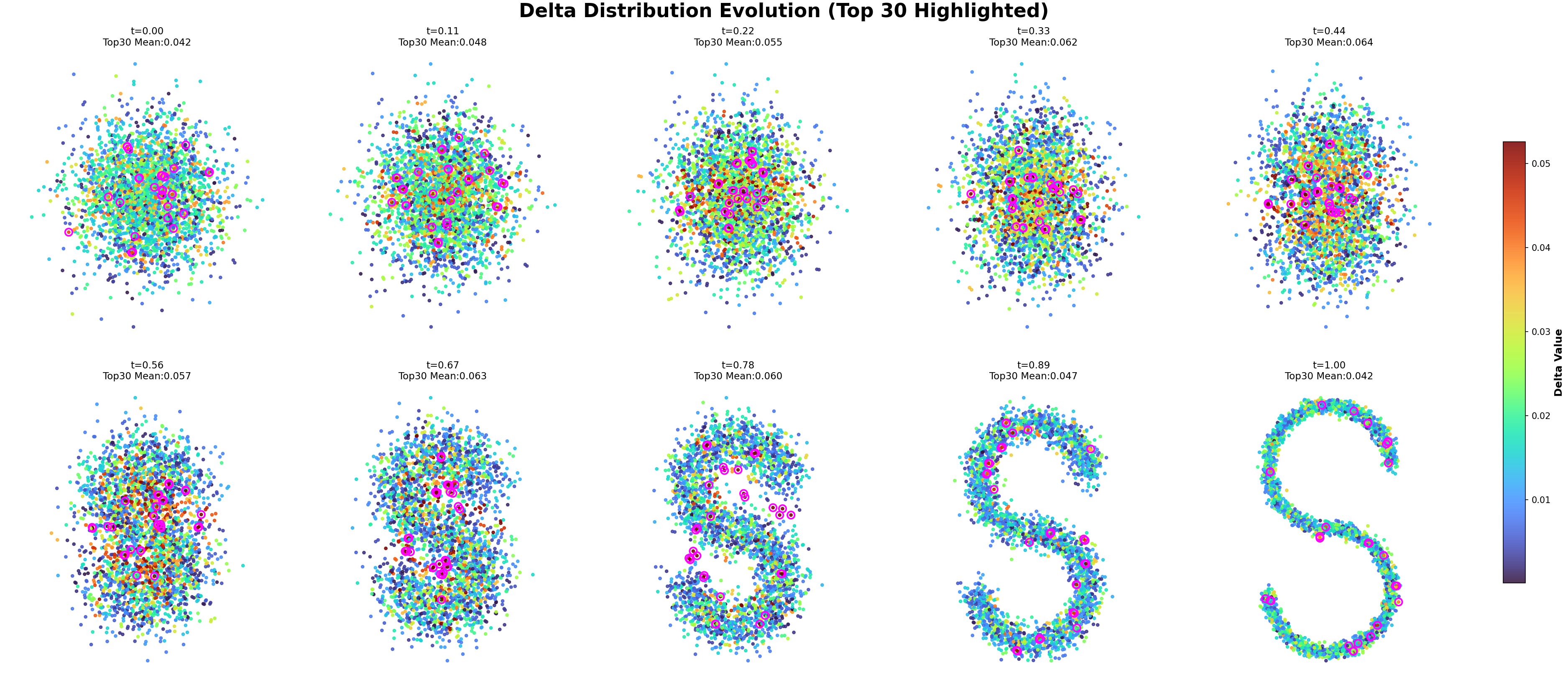}
    \caption{Delta ($\Delta$) Distribution Evolution. Points are colored by their $\Delta$ value (splitting intensity). The $x, y$ axes are spatial coordinates. Pink circles highlight the Top-30 particles concentrated at the flow's bifurcation neck.}
    \label{S 3000 delta}
\end{figure}

\begin{figure}[H]
    \centering
    \includegraphics[width=12cm, height=5cm]{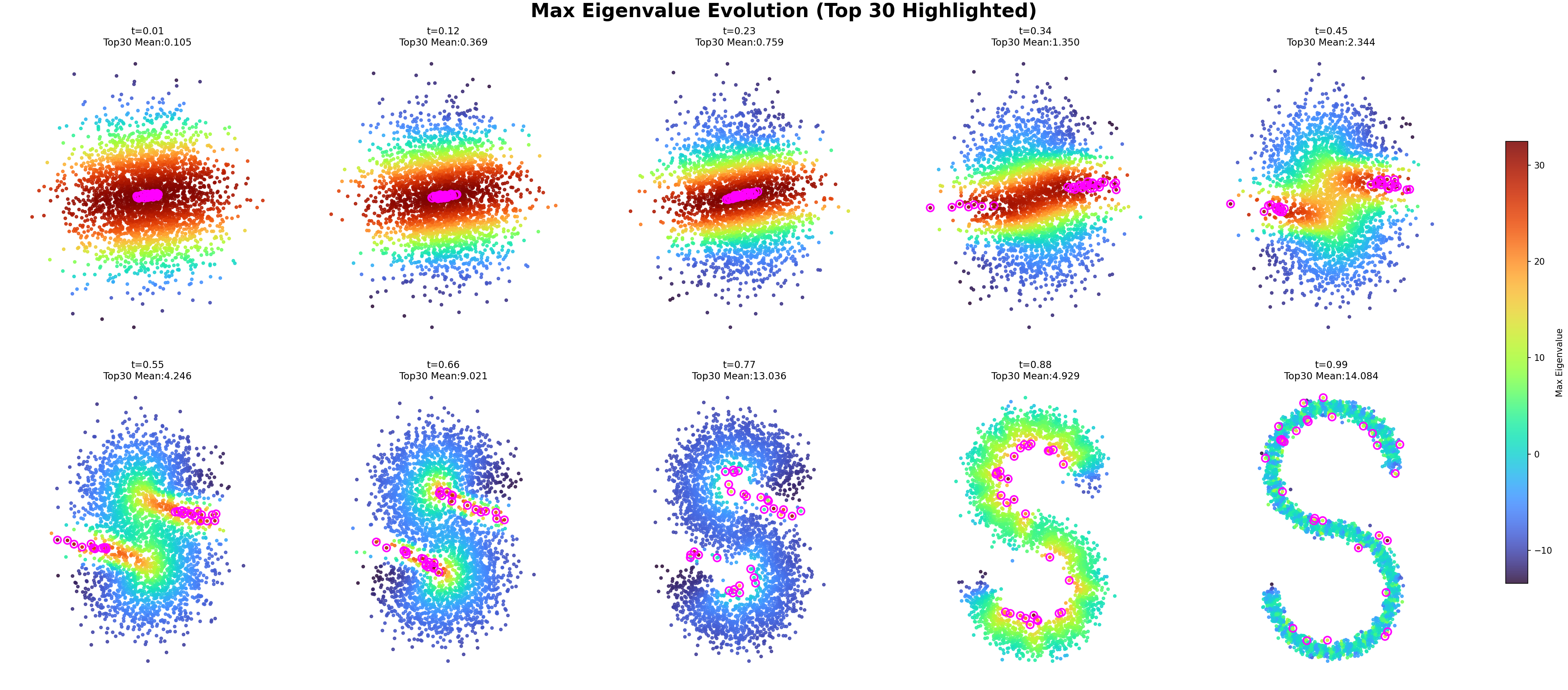}
    \caption{Max Eigenvalue Evolution. Points are colored by the Jacobian's max eigenvalue, indicating local manifold curvature. The $x, y$ axes are spatial coordinates, with pink circles tracking the Top-30 high-curvature hotspots.}
    \label{S 3000 eig}
\end{figure}

\begin{figure}[H]
    \centering
    \includegraphics[width=12cm, height=6cm]{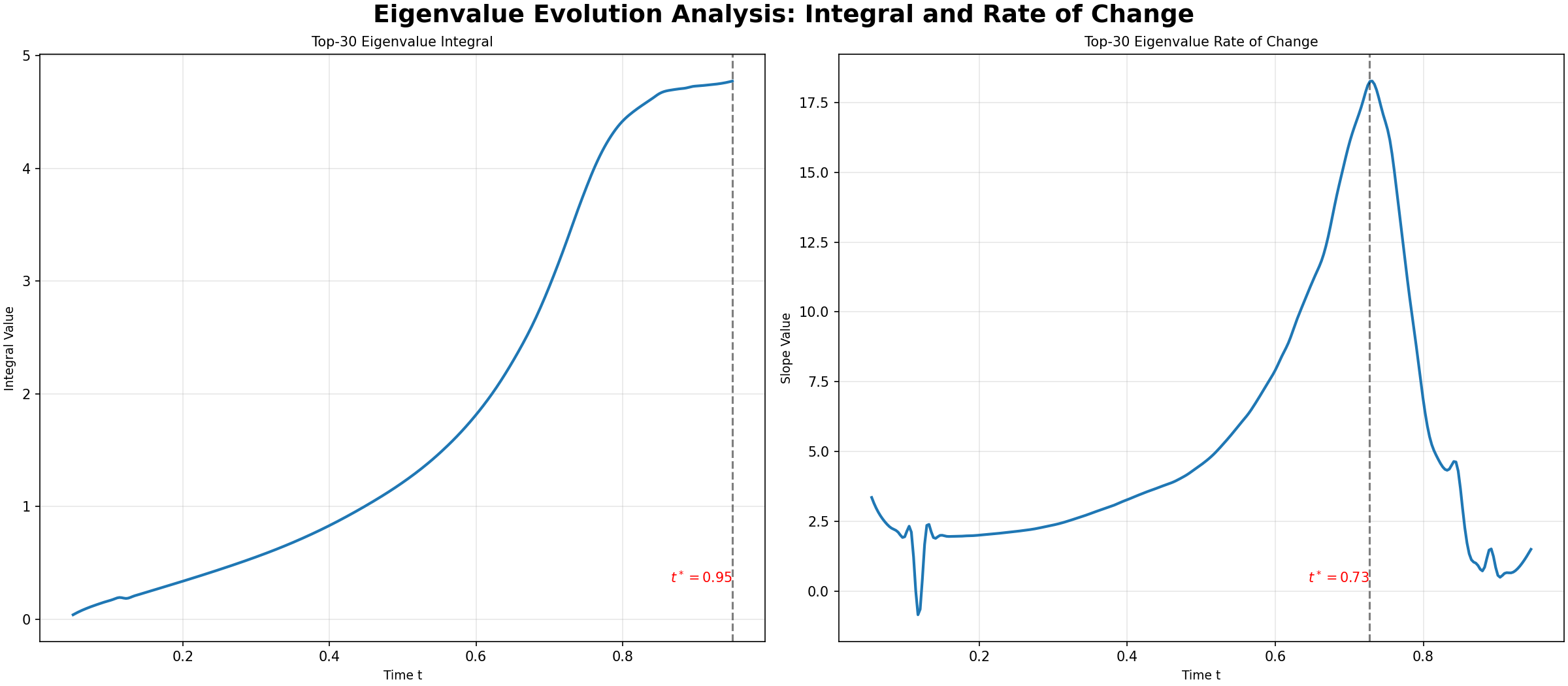}
    \caption{Eigenvalue Indicator Analysis. (Left) Cumulative integral of Top-30 eigenvalues. (Right) Rate of Top-30 eigenvalues' change (slope). The $x$-axis is time $t$; the vertical dashed line marks the peak $t^*$ signaling the onset of splitting.}
    \label{S 3000 indicator}
\end{figure}

\subsection{Gaussian Mixture}
The Gaussian Mixture dataset requires the probability mass to expand from a central Gaussian distribution into eight symmetric clusters. During this evolution, the transport process first breaks away from the central region and then gradually separates into multiple modes.

From the spatial maps, the early stage of the evolution shows Top-30 particles forming a high-intensity ring around the center. This ring marks the boundary where the flow begins to move away from the origin. As the evolution continues, the particle distribution further separates and gradually forms eight distinct regions that correspond to the target clusters in the Grid Evolution maps.

At the global level, we monitor the Top-30 eigenvalue integral and its slope throughout the evolution. The eigenvalue integral shows a sharp peak at time $t^*$, while the slope changes rapidly during the same interval. This behavior occurs during the early rupture at the center, before the eight clusters become clearly separated in the spatial maps. These indicator variations reveal the onset of the first bifurcation and provide an early signal of the subsequent multi-mode separation.

\begin{figure}[H]
    \centering
    \includegraphics[width=12cm, height=5cm]{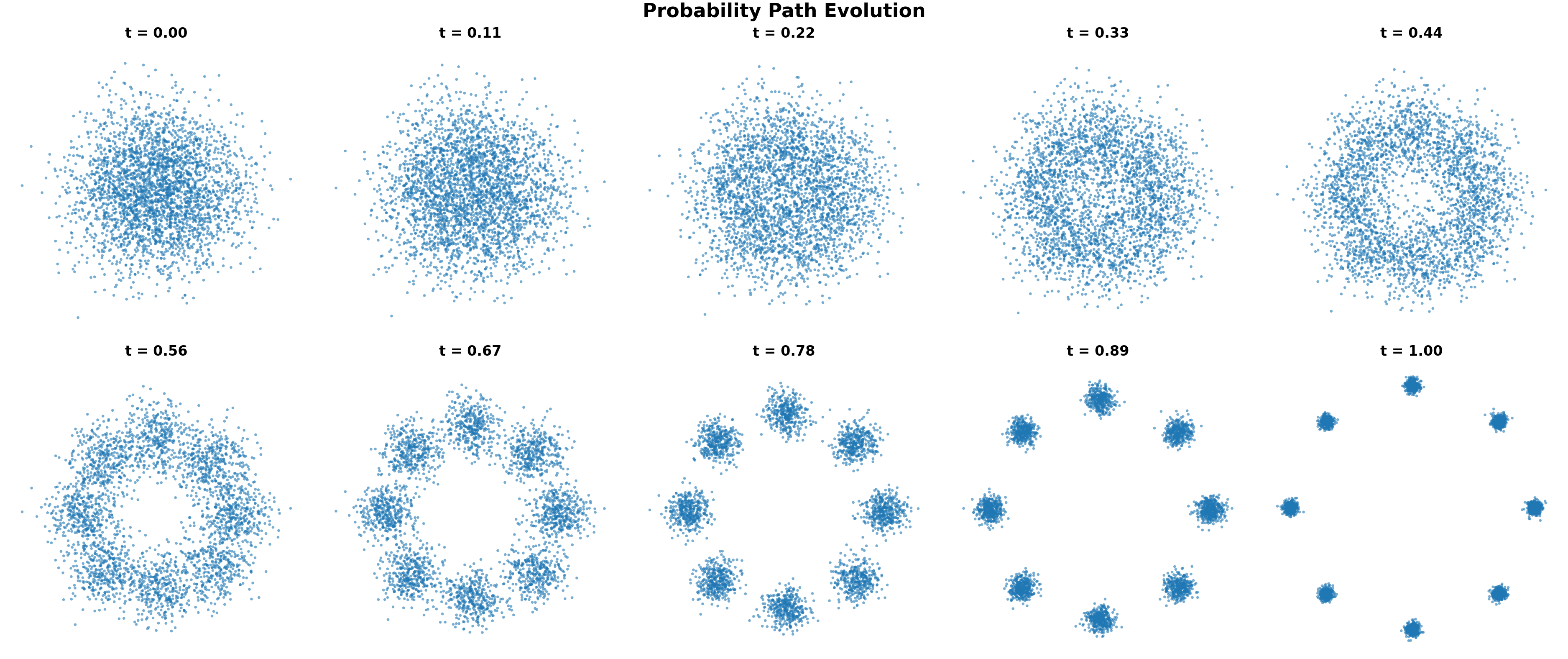}
    \caption{Probability Path Evolution. Snapshots of particle transport from a Gaussian prior to the Gaussian Mixture distribution. The $x, y$ axes represent 2D space, illustrating the structural splitting from $t=0.00$ to $t=1.00$.}
    \label{8 3000}
\end{figure}

\begin{figure}[H]
    \centering
    \includegraphics[width=12cm, height=5cm]{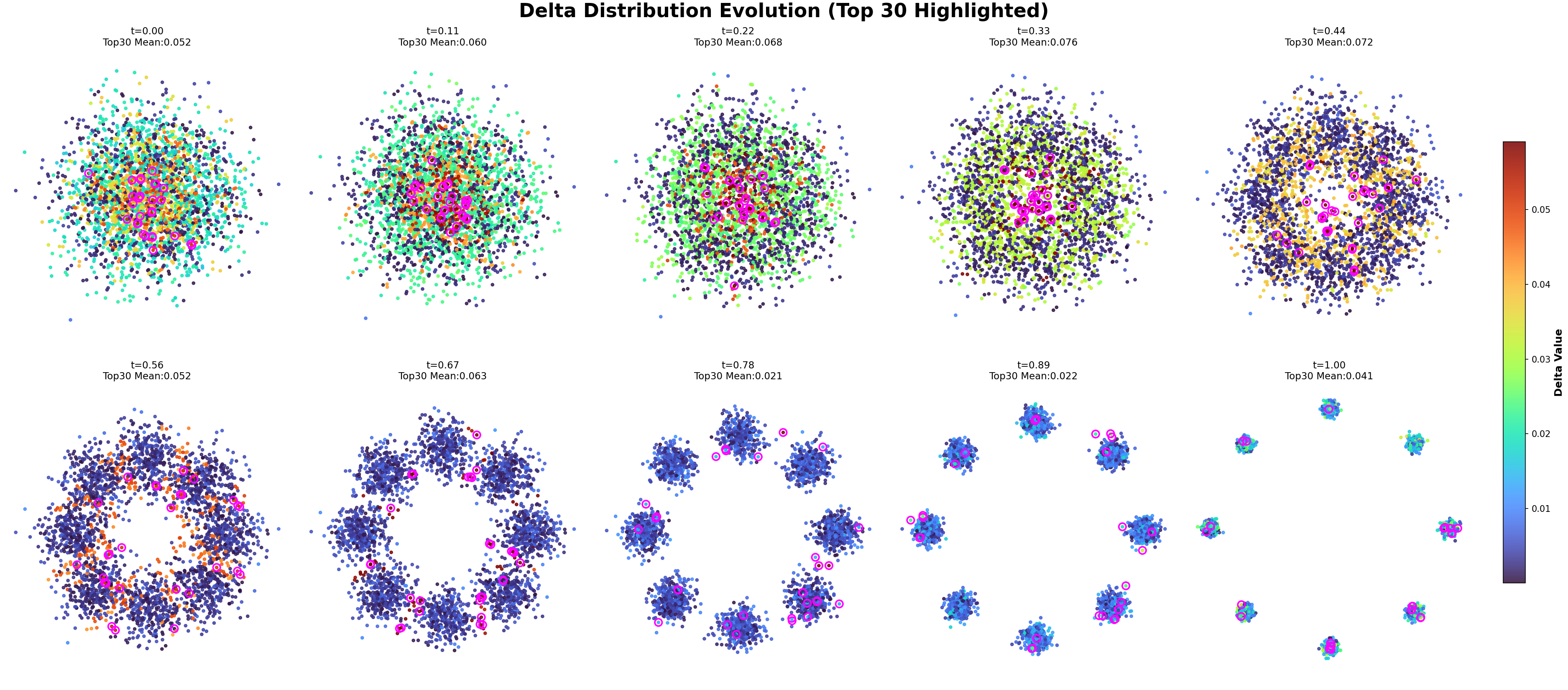}
    \caption{Delta ($\Delta$) Distribution Evolution. Points are colored by their $\Delta$ value (splitting intensity). The $x, y$ axes are spatial coordinates. Pink circles highlight the Top-30 particles concentrated at the flow's bifurcation neck.}
    \label{8 3000 delta}
\end{figure}

\begin{figure}[H]
    \centering
    \includegraphics[width=12cm, height=5cm]{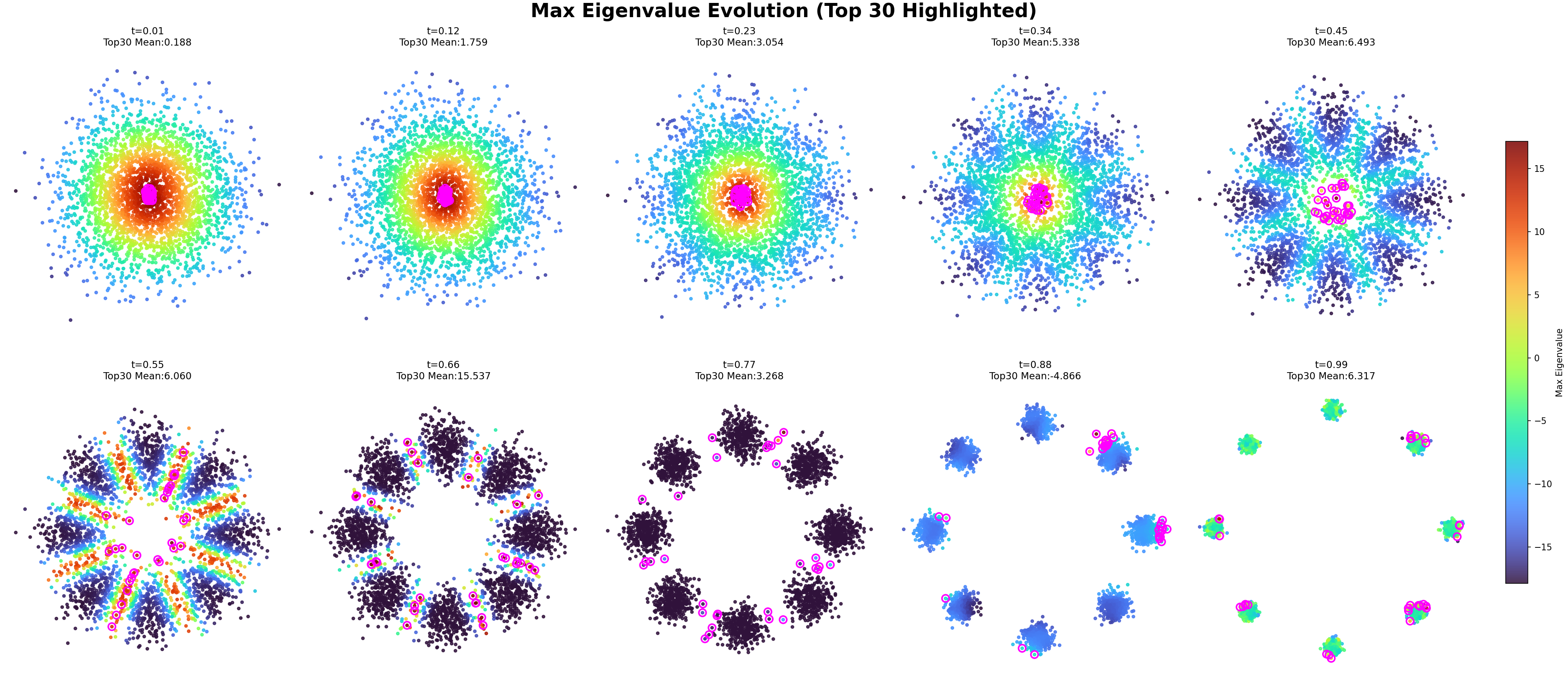}
    \caption{Max Eigenvalue Evolution. Points are colored by the Jacobian's max eigenvalue, indicating local manifold curvature. The $x, y$ axes are spatial coordinates, with pink circles tracking the Top-30 high-curvature hotspots.}
    \label{8 3000 eig}
\end{figure}

\begin{figure}[H]
    \centering
    \includegraphics[width=12cm, height=6cm]{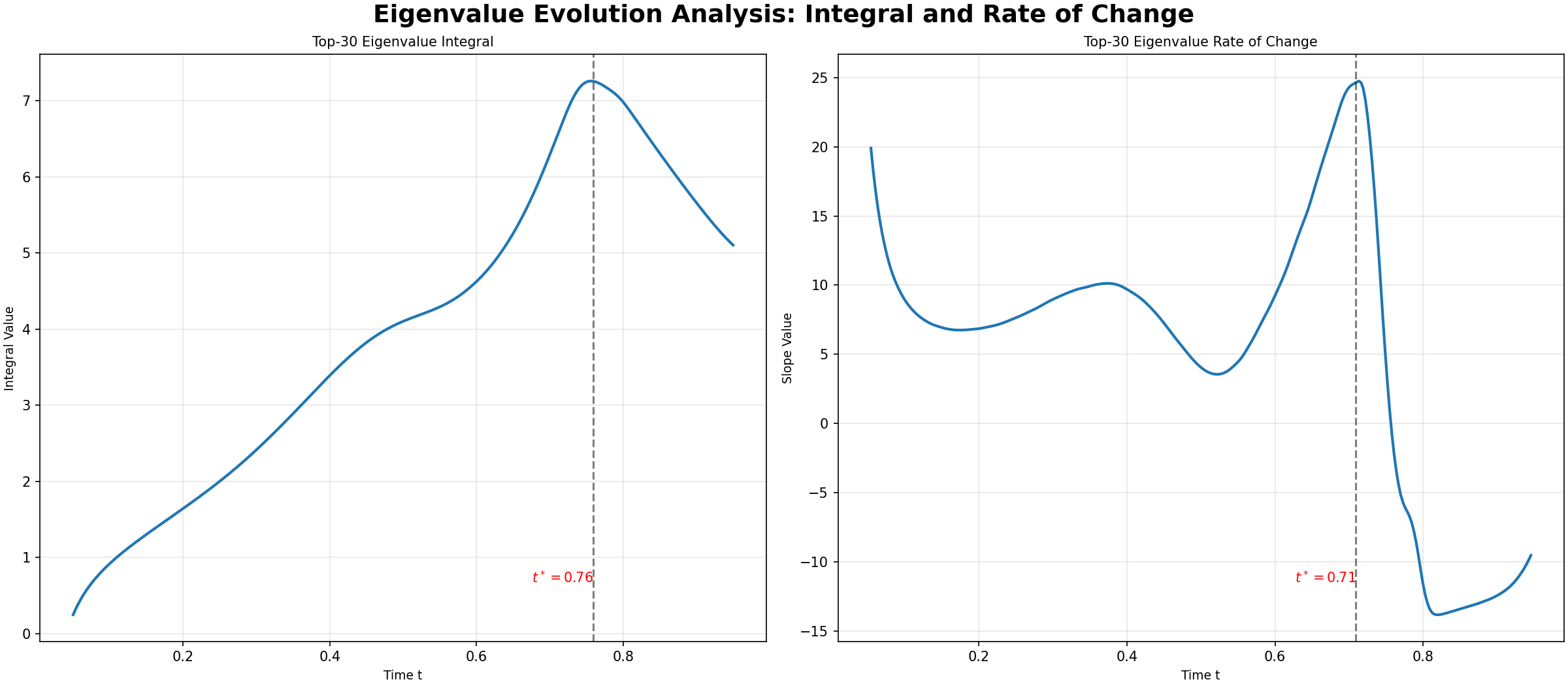}
    \caption{Eigenvalue Indicator Analysis. (Left) Cumulative integral of Top-30 eigenvalues. (Right) Rate of Top-30 eigenvalues' change (slope). The $x$-axis is time $t$; the vertical dashed line marks the peak $t^*$ signaling the onset of splitting.}
    \label{8 3000 indicator}
\end{figure}
 
\section{Conclusion}
We present a generative modeling framework based on stochastic interpolation. Deep neural networks are used to parameterize the time-dependent coefficients, enabling efficient transport from simple priors to complex data distributions.To address issues like mode collapse and topological changes, we examine the evolution of probability flows using differential geometry and dynamical systems. Two key indicators are introduced: the local velocity divergence ($\delta$) and the largest eigenvalues of the transport Jacobian. To avoid losing details from global averaging, we adopt a local Top-$K$ strategy to focus on regions with the strongest deformations and curvature.

Experiments show that these indicators provide early warnings for mode splitting in multi-modal distributions. For tasks involving radial divergence, grid rupture, or high-curvature manifolds, the peak times of the indicator integrals and their rates of change align closely with topological changes in the distribution. This demonstrates that our indicators help detect transport bottlenecks and offer insights into the stability and geometric evolution of generative ODEs. Overall, our work lays a foundation for building more interpretable and robust continuous-time generative models.

\bibliographystyle{plain}
\bibliography{reference}
\end{document}